\pdfoutput=1
    \documentclass[11pt,leqno]{article}%
    \usepackage{amsmath,amsfonts,amssymb,amsthm,fancyhdr}
    \usepackage{mathptmx,enumerate}%
\long\def\ignore#1{\par}

\newcommand\s{\par\smallskip}
\newcommand\n{\noindent}
\newcommand\sn{\s\noindent}

\renewcommand\b{\par\bigskip}

\renewcommand\newline{\hfil\break}

\newcommand{\qua}{\quad\kern-4truept}

\newcommand\stru{{\vrule height8pt depth4pt width0pt}}

\newcommand{\etab}{\boldsymbol{\eta}}
\newcommand\tiletab{\tilde\etab}

\newcommand{\qedsym}{$\blacktriangleleft$}
\renewcommand\qed{\nopagebreak\hspace*{1pt}\hfill\qedsym\par\s}
\makeatletter
\renewenvironment{proof}[1][\proofname]{\par
  \normalfont
  \topsep6\p@\@plus6\p@ \trivlist
  \item[\hskip\labelsep\scshape
    #1\@addpunct{:}]\ignorespaces
}{%
  \qed\endtrivlist
}

\theoremstyle{plain}

\newtheoremstyle{bookremark}{\topsep}{\topsep}{\normalfont}{}%
{\normalfont\itshape}{:}{.5em}{\thmname{#1}\thmnumber{#2}\thmnote{(\itshape{#3})}}

\newtheoremstyle{harmonicplain}{\topsep}{\topsep}{\itshape}{}{\bfseries}%
{.}{.5em }{\thmname{#1}\thmnumber{ #2}\thmnote{ (\itshape{#3})}}

\swapnumbers
\theoremstyle{harmonicplain}%
\newtheorem*{theorem}{\itshape Theorem}
\newtheorem*{lemma}{\itshape Lemma}

\newtheorem{stheorem}[subsection]{\  \itshape Theorem}

\theoremstyle{bookremark}%
\newtheorem*{remark}{ Remark}

\theoremstyle{bookdefinition}%
\newtheorem*{definition}{\normalfont\scshape Definition}

\makeatletter
\renewcommand\subsection{\@startsection{subsection}{2}{\z@}%
                                     {-3.25ex\@plus -1ex \@minus -.2ex}%
                                     {-1.5ex \@plus .2ex}%
                                     {\bfseries}}
\makeatother

\newcommand{\abs}[1]{\lvert #1 \rvert}

\newcommand{\norm}[1]{\ensuremath{\lVert #1 \rVert}}

\newcommand{\set}[2]{\ensuremath{\{#1 \,\mathbf{:}\, #2\}}}
\newcommand{\bigbra}[1]{\bigl( #1 \bigr)}

\newcommand\enumaba[1]{\begin{enumerate}[\itshape\bfseries a. ]#1\end{enumerate}\par}

\newcommand\osta[1]{O\left(#1\right)}

\newcommand\ostaB[1]{O\Bigl(#1\Bigr)}

\newcommand{\var}{\operatorname{\text{var}}}
\newcommand{\Lip}{\operatorname{\text{Lip}}\nolimits}

\newcommand{\vsave}[1][6]{\par\vskip -#1pt\noindent}
\newcommand\save{\par\vskip -19pt}

\newcommand\inv{^{-1}}
\newcommand{\half}{\ensuremath{\frac{1}{2}}}
\newcommand{\tenth}{\ensuremath{\frac{1}{10}}}

\newcommand\osc{{\footnotesize\textsf {osc}}}

\def\example{\smallskip\noindent{\textbf{Example.}\hspace{.69em} }}

\def\ve{\varepsilon}
\def\vt{\vartheta}
\def\vp{\varphi}

\newcommand\ce{{\mathcal E}}
\newcommand\cm{{\mathcal M}}
\newcommand\cl{{\mathcal L}}
\newcommand\ch{{\mathcal H}}

\newcommand{\bbr} {\mathbb{R}}
\newcommand{\bbn}{\mathbb{N}}
\newcommand{\bbz}{\mathbb{Z}}
\newcommand{\restr}[1]{\,\lower.53ex \hbox{\vrule depth1pt width0.4pt height10pt
             \lower.1ex\hbox{$\,#1$}}}

\newcommand{\hdim}[1]{\mathop{\ch\!\text{-}\mathrm{dim}}#1}
\newcommand{\lmdim}[1]{\mathop{\cl\cm\!\text{-}\mathrm{dim}}#1}

\newcommand{\mdim}[1]{\mathop{\cm\!\text{-}\mathrm{dim}}#1}

\newcommand{\bfit}[1]{\bfseries\itshape{#1}\normalfont}
\newcommand{\aba}{{\bfit {a.}\ }} 
\newcommand{\bab}{{\bfit {b.}\ }}
\newcommand{\cac}{{\bfit {c.}\ }}

    \title{Restrictions of continuous functions} %
    \author{Jean-Pierre Kahane and Yitzhak Katznelson}
\date{}

\begin{document}
\thispagestyle{empty}

    \maketitle

\section*{Introduction}
Given a continuous real-valued function on $[0,1]$, and a closed
subset $E\subset [0,1]$ we denote by $f\restr E$ the restriction
of $f$ to $E$, that is, the function defined only on $E$ that takes
the same values as $f$ at every point of $E$. The restriction
$f\restr E$ will typically be ``better behaved'' than $f$.
It may have bounded variation when $f$ doesn't,
 it may have
a better modulus of continuity than $f$, it may be monotone when
$f$ is not, etc. All this clearly depends on $f$ and on $E$, and the
questions that we discuss here are about the existence, for every
$f$, or every $f$ in some class,   of 
``substantial'' sets  $E$ such that $f\restr E$ has bounded 
total variation, is monotone, or satisfies a given modulus of continuity.
The notion of ``substantial'' that we use is that of either
   Hausdorff or  Minkowski dimensions,
both are defined below.

Here is an outline of the paper. We refer
to theorems by the subsection  in which they are stated.

Section 2  deals with restrictions of bounded variation.
Theorem 2.1, part {\bfit I} states that every continuous 
real-valued function  on $[0,1]$ has bounded variation on some set
of Hausdorff dimension 1/2. 
Part {\bfit {II}} of the theorem  shows that this is optimal by  
constructing 
an appropriate lacunary series whose sum
has unbounded variation on every closed set
of Minkowski dimension bigger than $1/2$ (and hence on every
set of Hausdorff dimension bigger than $1/2$). 
Analogous results 
for $\bbr^{d}$-valued functions are proved in subsection 2.6.

Section 3 deals with restrictions that satisfy a
H\"older condition with parameter $\alpha\in (0,1)$.
It was known, though never stated
in this form,  that for every 
continuous function $f$ on $[0,1]$
and every $\alpha\in (0,1)$ there exists  sets $E$ of 
Hausdorff dimension
$1-\alpha$ such that $f\restr E$ satisfies a H\"older $\alpha$
condition (see subsection 3.1). 
Extending the methods used in the proof of theorem 2.1, we  give an elementary proof of the result (theorem 3.1 part {\bfit I})
and  show, in part {\bfit {II}}, that it is optimal by constructing,
as in the proof of part {\bfit {II}} of theorem 2.1,  an 
approriate lacunary series whose sum is
a function for which nothing better can be done.

In section 4, theorem 4.1, we construct continuos functions $f$ 
that satisfy a H\"older-$\alpha$ condition for all $\alpha<1$
and yet
 if  $f\restr E$ is Lipschitz or monotone, then $E$ 
 is ``arbitrarily thin''.
Theorem 4.2 deals with monotone restrictions of continuous
functions.

 In section 5  we consider the relative advantage of restrictions
of functions that satisfy  various  H\"older smoothness conditions,
give  partial results and
 point out some open problems.

By including the short section 1, we try to make the exposition  
elementary  and self-contained, requiring no  background 
material beyond what should be ``commonly known''.

\subsection*{Notations and terminology. }{\qua}

\emph{A modulus of continuity}   
 is a monotone increasing continuous concave function  $\omega(t)$ 
  on $[0,1]$,  such that $\omega(0)=0$.

Given a real-valued function $f$ on $[0,1]$, a closed set $E$, and
a modulus of continuity $\omega$, we write
$f\restr E\in C_{\omega}$ if for all $t\in E$ there exist
$\delta=\delta(t)>0$ and $C=C(t)$ such that if $\tau\in E$ and
$\abs{t-\tau}\le\delta(t)$ then
$\abs{f(t)-f(\tau)}\le C(t)\omega(t-\tau)$.

For $\omega(t)=t^{\alpha}$, $0<\alpha\le 1$ we write
$\Lip_{\alpha}$ instead of $C_{\omega}$. $\Lip_{1}$ is 
usually referred to as the \emph{Lipschitz class}, 
while $\Lip_{\alpha}$, $0<\alpha<1$, 
as the H\"older $\alpha$ class.\footnote{Some classics refer to
the H\"older classes as \emph{the Lipschitz $\alpha$ classes}
---hence the notation.} 

The \emph{(total) variation}, $\var(E,f)$, of a function $f$ on a closed
set $E$, is defined by
\begin{equation*}
\var(E,f)=\sup \sum\abs{f(x_{j+1})-f(x_{j})},
\end{equation*}
the supremum is for all finite monotone increasing sequences 
$\{x_{j}\}\subset E$.
We  write $f\in BV(E)$ if $\var(E,f)<\infty$.

The oscillation of $g$ on a set $E$ is
\begin{equation}
\osc(g,E)=\max_{x\in E}g(x)-\min_{x\in E}g(x).
\end{equation}

Finally, if $E\subset[0,1]$ is closed, we denote by $\abs{E}$
the (Lebesgue) measure of $E$.

\section{Dimensions }

\subsection{(Lower) Minkowski  dimension. }{\qquad}\\

\vsave[15]
\begin{definition} Let $s>0$.
An \emph{$s$-separated set of length $m$} is
a set $J=\{x_{j}\}_{j=1}^{m}$
in $[0,1]$ such that $\abs{x_{k}-x_{j}}>s$ for $j\ne k$. 
\end{definition}

For a subset $E\subset [0,1]$, denote  by $L_{n}(E)$  the smallest number 
of intervals of length $n^{-1}$ needed to cover  $E$. 
Denote by
 $L_{n}^{*}$  the largest number $L$
such that $E$ contains some $n\inv$-separated sequence of length 
$L$.
\begin{lemma}
\begin{equation}\label{mink1}
L_{n}(E)\le L^{*}_{2n}(E)\le L_{2n}(E). 
\end{equation}  

\end{lemma}
\begin{proof} A pair of points whose distance is $>(2n)\inv$ cannot belong
to the same interval of length $(2n)\inv$. Conversely, if 
$\{x_{j} \}_{j=1}^{L_{n}^{*}}$ is 
a maximal $(2n)\inv$ separated subset of $E$, then
the intervals of length $n\inv$ centered at $x_{j}$ cover $E$.
\end{proof}

The Minkowski dimension, $\mdim (E)$ of  $E$
is defined as the limit, if it exists,
\begin{equation}\label{mink2}
\mdim(E)=\lim_{n\to\infty} \frac{\log L_{n}(E)}{\log n}=
\lim_{n\to\infty} \frac{\log L_{n}^{*}(E)}{\log n}.
\end{equation}

The lower Minkowski dimension $\lmdim (E)$ of  $E$
is well defined for all sets by 
\begin{equation}
\lmdim (E)=\liminf \frac{\log L_{n}(E)}{\log n}=
\liminf \frac{\log L_{n}^{*}(E)}{\log n}.
\end{equation}

\example
If $E=\{\frac{1}{j}\}_{j=1}^{\infty}$, the subset 
$\{\frac{1}{j}\}_{j=1}^{{n}}$ is $n^{-2}$ separated and
$L_{n^{2}}^{*}(E)\ge n$. On the other hand
the intervals $[jn^{-2},(j+1)n^{-2}]$, $j=1,\dots,{n}$ 
cover $\{\frac{1}{j}\}_{j={n}}^{\infty}$, and ${n}$ additional
intervals of the same size cover
$\{\frac{1}{j}\}_{j=1}^{{n}}$, so that $L_{n^{2}}(E)\le 2n$.
By \eqref{mink1} $L_{n^{2}}(E)\sim n$, the limit in \eqref{mink2}
exists, and
$\mdim(E)=\half$.

\subsection{Hausdorff dimension.  }\label{Hausdorff}
The Hausdorff dimension 
$\hdim(E)$ of
a set $E\subset \bbr$ is  the infimum of the numbers $ c$ 
for which there is a constant $C$ such that, for every $\ve>0$, there
exists a covering of $E$ by intervals $I_{n}$ satisfying:
\begin{equation}\label{hc}
\sup_{n} \abs{I_{n}}<\ve \qua \text{and} \qua \sum \abs{I_{n}}^{c}<C.
\end{equation}
Since covering by intervals of arbitrary lengths $\le \ve$ can be more 
efficient than covering by intervals of a fixed length, 
\begin{equation}
\hdim(E)\le \lmdim(E);
\end{equation}
the Hausdorff dimension of a set $E$
 is bounded above by its lower Minkowski dimension. 
 The inequality can be strict:  for example,
 if $E$ is countable then $\hdim(E)=0$, while
 $\lmdim(E)$ can be as high as 1.
 
 A useful criterion for a lower bound on the Hausdorff dimension
 of a closed set $E$ is the following:
 \begin{lemma}
Assume that $E$ carries a probability measure $\mu$
such that $\mu(I)\le C\abs{I}^{\delta}$ for every interval $I$
then $\hdim(E)\ge \delta$.
\end{lemma}
 \begin{proof}
If $c<\delta$, and $I_{n}$ are  intervals such that $\abs{I_{n}}<\ve$  
and $\cup I_{n}\supset E$,   then
\begin{equation}
1\le \sum \mu(I_{n})\le C\sum\abs{I_{n}}^{\delta}\le C\ve^{\delta-c}\sum\abs{I_{n}}^{c}.
\end{equation}
That means $\sum\abs{I_{n}}^{c}>C\inv\ve^{c-\delta}$ which is unbounded as $\ve\to 0$.
\end{proof}

\subsection{Determining functions.}
A \emph{ Hausdorff  determining function} is a continuous nondecreasing function  $h$ on $[0,1]$ satisfying $h(0)=0$. The
Hausdorff dimension introduced in the previous subsection uses
explicitly, in \eqref{hc}, the functions $h_{c}(t)=t^{c}$, with $0<c\le1$ 
as does (implicitly) the definition of the Minkowski dimension.

A set $E\subset [0,1]$ has zero $h$-meassure if, for every $\ve>0$, there exist
intervals $I_{n}$ such that $\sum h(\abs{I_{n}})<\ve$ 
and $E\subset \cup I_{n}$.

A set $E\subset [0,1]$ is 
Minkowski-$h$-null if $\liminf L_{n}h(1/n)=0$. 

A set that is Minkowski $h$-null has zero $h$-measure. 
The converse is false.

\section{Restrictions of  Bounded Variation}

\subsection{The total variation of restrictions. }\label{mainvar}

    Given a function $f$ on $\bbr$ and a closed set $E$, 
we denotes the total variation of the restriction  $f\restr E$ 
of $f$ to $E$ by $\var(E,f)$,
and write $f\in BV(E)$ if $\var(E,f)<\infty$.

\begin{theorem} \textbf{I}: \ 
For every real-valued $f\in C([0,\,1])$, there are closed 
sets $G\subset [0,\,1]$,
such that $\hdim(G)\ge\half$  and $f\in BV(G)$.

\s
\textbf{II}:  \  
 There exists real-valued  functions $F\in C([0,1])$
such that $\var(E,F)=\infty$ for every closed set
 $E\subset [0,1]$ such that $\lmdim(E)> \half$,
(and, in particular, for closed sets $E$ such that $\hdim(E)>\half$). 
\end{theorem}

\subsection{}\label{procedure1}
The proof of part \bfit{I}  of the theorem uses the following lemma.
\begin{lemma} Let $I$ be an interval and
 $E\subset I$ a closed set, $\vp\in C(E)$
and $\osc(\vp,E)=a$. Then there are  subsets
$E_{j}\subset E$, $j=1,2$, carried by disjoint intervals $I_{j}$,
such that $\abs{E_{j}}\ge\frac{1}{4}\abs{E}$
and $\osc(\vp,E_{j})\le \frac{a}{2}$.
\end{lemma}
\begin{proof}
If $I=[t_{1},t_{2}]$ let $t_{3}$ be such that 
$\abs{E\cap [t_{1},t_{3}]}=\half\abs{E}$. Set $I_{1}=[t_{1},t_{3}]$
and $I_{2}=[t_{3},t_{2}]$.

Define $E_{1}\subset I_{1}$ as follows:
Let $[c,c+a]$ be the smallest interval containing $\vp(E\cap I_{1})$. Write
$G_{1}=E\cap\vp\inv([c,c+\half a])$ and 
$G_{2}=E\cap\vp\inv([c+\half a,c+a])$, 
and observe that either $\abs{G_{1}}\ge \half\abs E$
or $\abs{G_{2}}\ge \half\abs E$ (or both).
 Set $E_{1}$ as  $G_{1}$ in the first case, and as $G_{2}$ otherwise. 
 Define $E_{2}\subset I_{2}$ in the same way.
\end{proof}
We call the sets $E_{j}$ \emph{descendants} of $E$, and refer to
the replacement of each $E$ by its two descendants as the 
\emph{standard procedure}. We sometime use the \emph{alternate
procedure } in which we replace each  $E$ by only one of the two
descendants.

\begin{proof}[Proof of the theorem, part \bfit{I}  ]
Let $f\in C([0,1])$ be real-valued. We apply the lemma, 
with $\vp=f$, repeatedly. We use the standard
procedure most steps  and the alternate procedure 
occasionally, $c(k)\sim 2\log_{2} k$ times out of $k$. 
After  $k$ iterations  we have a set $\ce_{k}$ which is the union
of
$2^{k-c(k)}\sim 2^{k}k^{-2}$ sets $E_{k,\alpha}$, each of Lebesgue measure $\ge 2^{-2k}$, carried
by disjoint intervals $I_{k,\alpha}$, and such that 
$\osc(g,E_{k,\alpha})\le 2^{-k}$.
Write $G=\bigcap_{k}\ce_{k}$.

For $x,y\in G$ let $k(x,y)$ be the last $k$ such that $x$ and $y$
are in the same component $E_{k,\alpha}$. Remember that
$\abs{f(x)-f(y)}\le 2^{-k}$. 

In  a monotone sequence $\{x_{j}\}_{j=1}^{N}\subset G$  and 
any $k\in\bbn$,
there are at most $2^{k-c(k)}\sim 2^{k}k^{-2}$ values of $j$
for which $k(x_{j},x_{j+1})=k$; so that 
\begin{equation}
\sum\abs{f(x_{j+1})-f(x_{j})}\le \sum 2^{k-c(k)}2^{-k}\sim \sum 2^{k}k^{-2}2^{-k}=\sum k^{-2}.
\end{equation}
It follows that the total variation of
$f\restr{G}$ is    bounded by $\sum k^{-2}$.

 Let $\mu_{k}$ a probability measure
carried by $\ce_{k} $ that puts the same mass $2^{c(k)-k}$  on every
$E_{k,\alpha}$. Observe that, for all $l\in\bbn$,
$\mu_{k+l}(E_{k,\alpha})=\mu_{k}(E_{k,\alpha})$.

Let $\mu$ be a weak-star limit of $\mu_{k}$ as 
$k\to\infty$. Since every interval $I$ of length $2^{-2k}$ intersects
at most two sets of the form
$E_{k,\alpha}$ we have $\mu(I)\le C\abs{I}^\frac{k-c(k)}{2k}$ and, by lemma
\ref{Hausdorff} 
$\hdim{G}\ge 1/2$.
\end{proof}

\subsection{}
The proof of part \bfit{II} of the theorem
is a construction that uses as a building block 
the  2-periodic  function $\vp$, defined by:
\begin{equation}\label{varphi}
\vp(2m+x)=1-\abs{x}\qquad\text{for  $\abs{x}\le 1$ and $m\in\bbz$.}
\end{equation}
 We  write $\vp_{n}(x)=\vp(2nx)$.

\begin{lemma}\label{basic_{d=1}}
Let $J=\{x_{j}\}\subset [0,1]$ be %
an $s$-separated monotone sequence of length $m$. If $m>2n$, 
then, for $a>0$,
\begin{equation}\label{basic0}
\var(J,a\vp_{n})=
\sum \abs{a\vp_{n}(x_{j+1})-a\vp_{n}(x_{j})}\ge (m-2n)2nas.
\end{equation}
\end{lemma}

\begin{proof}
There are at most $2n$ values of $j$ for which
$x_{j}$ and $x_{j+1}$ are separated by some $\frac{\ell}{2n}$, ($\ell=1,\dots,2n$).
For all other $j$ we have $a\vp_{n}$ linear and 
$\abs{a\vp_{n}'}=2an$ in $[x_{j},x_{j+1}]$ so that
\begin{equation}
\abs{a\vp_{n}(x_{j+1})-a\vp_{n}(x_{j})}= 2na(x_{j+1}-x_{j})\ge 2nas,\\
\end{equation} and there are at least $m-2n$ such values of $j$.
\end{proof}

\subsection{}
We can modify $a\vp_{n}$ somewhat without affecting \eqref{basic0}
materially.
\begin{lemma}
Let $g\in C([0,1])$,  $\norm{g}_{\infty}<nsa/10$,
and $G\in C([0,1])$ with  Lipschitz constant
bounded by $\frac{na}{10}$,  then
 \begin{equation}\label{basic11}
\var(J,G+a\vp_{n}+ g)\ge (m-2n)nsa.
\end{equation}
\end{lemma}
\begin{proof}
For the 
values of $j$ for which
$x_{j}$ and $x_{j+1}$ are \emph{not} separated by some 
$\frac{\ell}{2n}$
 we have
\begin{equation}
\begin{split}
\abs{a\vp_{n}(x_{j+1})-a\vp_{n}(x_{j})}&= 2na(x_{j+1}-x_{j}),\\
\abs{G(x_{j+1})-G(x_{j})}&\le \frac{na}{10} (x_{j+1}-x_{j}),\\
\abs{g(x_{j+1})-g(x_{j})}&\le \frac{nas}{5}\le \frac{na}{5} (x_{j+1}-x_{j}),
\end{split}
\end{equation}
so that
\begin{equation*}
\abs{G+a\vp_{n}+g)(x_{j+1})-
(G+a\vp_{n}+g)(x_{j})}\ge (2na-\frac{na}{10})(x_{j+1}-x_{j})-
\frac{nsa}{5}>nsa
\end{equation*}
which implies  \eqref{basic11}
\end{proof}

We use the lemma with $m=20n$ and the right-hand sides
of \eqref{basic0} and \eqref{basic11} will be (wastefully)
written simply as $n^{2}as$.
 
 \subsection{\hspace*{1pt} }
 \begin{proof}[Proof of  theorem 2.1, part \bfit{II} ]
For sequences $\{a_{l}\}$, $a_{l}>0$, and
$\{n_{l}\}\subset \bbn$ 
 write: $m_{l}=20 n_{l}$,  $s_{l}= n_{l}^{-2}\log n_{l}$, and
\begin{equation}
F=\sum_{l=1}^{\infty} a_{l}\vp_{n_{l}},\qquad
G_{k}=\sum_{l=1}^{k-1} a_{l}\vp_{n_{l}},\qquad
g_{k}=\sum_{l=k+1}^{\infty} a_{l}\vp_{n_{l}},
\end{equation}

\s
The sequences $\{a_{l}\}$, $a_{l}>0$ and
$\{n_{l}\}\subset \bbn$  are chosen (below)
 so that 
 
\enumaba{  
\item  $a_{k}\log n_{k}>k$,
\item  $\sum_{l=1}^{k-1} a_{l}n_{l}< \tenth a_{k}n_{k} $
\item $\sum _{l>k}a_{l}< \tenth n_{k}a_{k}s_{k}$.
}
\sn
These conditions guarantee that the lemma applies 
with $n=n_{k}$, $G=G_{k}$ and $g=g_{k}$
so that if $J$ is $s_{k}$ separated of length $m_{k}$, then
\begin{equation}
\var(J,F)\ge n_{k}^{2}a_{k}s_{k}=a_{k}\log n_{k}>k.
\end{equation}

If $\lmdim(E)>1/2$ then, for all 
$k>k(E)$, $E$ contains $s_{k}$-separated sequences 
$J_{E}(n_{k})$ of length $m_{k}$, so that
\begin{equation}
\var(E,F)\ge \var(J_{E}(n_{k}),F)>k,
\end{equation} 
and
the
function $F=\sum_{l=1}^{\infty} a_{l}\vp_{n_{l}}$ 
has infinite variation on every closed $E$
such that $\lmdim(E)>\half$. 
 
The sequences 
$\{a_{l}\}$ and
$\{n_{l}\}$ 
are defined recursively:

 Take  $a_{1}=1/2$  and $n_{1}=100$.
 
 If $a_{l}$ and $n_{l}$ defined for $l\le k$, 
 set $a_{k+1}=\frac{1}{20} a_{k}n_{k}^{-1}$,\qua and observe that
 this rule guarantees that
$\sum_{j>k} a_{j}<2 a_{k}$, so that 
\cac \ is satisfied.  

 Now take $n_{k+1}$ big enough to satisfy 
conditions  \aba and \bab \ 
\end{proof}

\subsection{ $\bbr^{d}$-valued functions.} \label{multi}

The generalization of Theorem \ref{mainvar} to $\bbr^{d}$-valued functions is the following statement:

\begin{theorem} \textbf{I}: \ 
For every continuous $\bbr^{d}$-valued  
function $g$, there are closed 
sets $E\subset [0,\,1]$,
such that $\hdim(E)\ge\frac{1}{d+1}$  and $g\in BV(E)$.

\textbf{II}: \  
There exists continuous $\bbr^{d}$-valued
functions $F$
such that if $E\subset [0,1]$ is closed and $\lmdim(E)> \frac{1}{d+1}$
then $\var(E,F)=\infty$.

\end{theorem}

The proofs of both parts are the obvious variations on the proofs
for $d=1$. 

The proof of part \textbf{I} differs from that of the corresponding
part of Theorem \ref{mainvar} only in the estimate of the measures
of the
sets $E_{k,\alpha}$ defined at the $k$'th stage,  carried, as before,
by disjoint intervals $I_{k,\alpha}$, and such that 
$\osc(g,E_{k,\alpha})\le 2^{-k}$, but
now of Lebesgue measure $\ge 2^{-(d+1)k}$. This guarantees that
the Hausdorff dimension of the set, constructed as before,
is $\ge\frac{1}{d+1}$.

For part \bfit{II} we replace the function $\vp_{n}$ 
by 
$\psi_{n}=\psi_{n,d}(mx)$ where $m=[n^{1/d}]$  (the integer part 
of $n^{{1}/{d}}$) and
$\psi_{n,d}$ is an even $2$-periodic $\bbr^{d}$-valued 
 function satisfying:   
 $\norm{\psi_{n,d}}\le 1$
 and,  for $x,y$  such that  $[x]=[y]$ and $\abs{x-y}\ge 1/n$:
\begin{equation}\label{315}
\norm{\psi_{n,d}(x)-\psi_{n,d}(y)}\ge
n^{-\frac{1}{d}} 
\end{equation}
so that
\begin{equation}\label{316}
\norm{\psi_{n}(x)-\psi_{n}(y)}\ge n^{-\frac{1}{d}} 
\qua\text{if}\qua  [mx]=[my]\qua\text{and}\qua
\abs{x-y}\ge n^{-\frac{d+1}{d}}. 
\end{equation}
A set $E$ such that 
$\lmdim (E)>\frac{1}{d+1}$, $E$ contains, when $n$ is large, 
$n^{-\frac{d+1}{d}}$-separated
sequences $\{x_{j}\}_{1}^{L}$ of length $L>>n^{\frac{1}{d}}$ 
and for all, but at most $m\sim n^{\frac{1}{d}}$ values of $j$,
we have
$\norm{\psi_{n}(x_{j+1})-\psi_{n}(x_{j})}\ge n^{-\frac{1}{d}}$
so that
the variation of $\psi_{n}$ on $E$ is large.

 One can construct the functions $\psi_{n,d}$ as follows.
 Let %
 $A_{m}=A_{m,d}$ be the set of $(m+1)^{d}$ points 
 $v_{l}=(k_{1},\dots k_{d})$ satisfying $0\le k_{j}\le m$
 in $\bbn^{d}$, enumerated in a way that $\norm{v_{l+1}-v_{l}}=1$, i.e.,
 $v_{l}$ and $v_{l+1}$ have the same entries  except for one, on which  they differ by 1. \ 
 The function $\psi_{n,d}$
 is defined on $[-1,1]$ by stipulating that it  is 2-periodic, 
 even, and it maps
 $[\frac{l}{(m+1)^{d}},\frac{l+1}{(m+1)^{d}}]$ linearly onto 
 $[\frac{v_{l}}{m},\frac{v_{l+1}}{m}]$.

\s

\section{
H\"older restrictions}\label{holderrest}
\begin{stheorem}
\textbf{I}: \ 
Assume $0<\alpha<1$. Given a continuous function $f$, there exists a closed set $E$ such that $\hdim E=1-\alpha$, and $f\restr E\in\Lip_{\alpha}$.

\textbf{II}: \ For $0<\alpha<1$
there exist continuous functions $f$ such
that if $f\restr E\in \Lip_{\alpha}$ for a closed set $E$,
then 
$\hdim E \le 1-\alpha$. 
\end{stheorem}\label{3.1}

Part \bfit{I} of the
theorem derives easily from properties of Gaussian stationary processes
on the circle, established in  \cite{jpkrandom}. The proof reads:

``Take a Gaussian stationary process $X$ on the circle (Fourier series with independent Gaussian coefficients) such that  
$X\in\Lip_{\alpha} $ and $\hdim {X^{-1}(0)}=\alpha$ a.s. Then write 
$E=(X-f)\inv(0)$ and apply remark 2 in Chapter 14, section 5, 
page 206 of \cite{jpkrandom}.''

Part \bfit{II} of the theorem shows that part \bfit{I} is optimal.
We give here an elementary proof of 
both parts.

\subsection{}\label{proceduregen}
We prove part {\bfit I} of the theorem by the method used in
the proof of part {\bfit I} of theorem \ref{mainvar}. The following is an
extension of the procedures introduced in \ref{procedure1}.

\begin{lemma} Let 
 $E\subset I\subset [0,1]$ be a closed set, $f\in C_{\bbr}(E)$
and $\osc(f,E)=a$. Given $\ve>0$, integers $k\ge 2$ and $l\ge 2$, 
 there are  subsets
$E_{m}\subset E$, $m=1,2,\dots, k$, carried by disjoint intervals 
$I_{m}$,
such that 
\begin{enumerate}[\bf a. ]
\item The distance between any two $E_{m}\stru's$ is at least $\abs{E}\ve/k$;
\item $\abs{E_{m}}\ge\frac{1-\ve}{kl}\abs{E}$;
\item
$\osc(f,E_{m})\le \frac{a}{l}$.
\end{enumerate}
\end{lemma}
\begin{proof}
Choose the increasing sequence $\{x_{m}\}$, $m=0,\dots,k$ 
so that 

\centerline{${\abs{E\cap [0,x_{m}]}=\abs{E}\frac{m}{k}}$,} 
\n and
let $y_{m}=x_{m}+\abs{E}\ve/k$. Write $I_{m}=[y_{m},x_{m+1}]$ and
$\tilde E_{m}=E\cap I_{m}$. 

Then 
$\abs{\tilde E_{m}}\ge \abs{E}\frac{1-\ve}{k}$.

\s
Let $J=[\min_{x\in E} f(x),\max_{x\in E} f(x)]$ (so that $\abs{J}=a$).
Divide $J$ into $l$ equal intervals, $J_{s}$, $s=1,\dots,l$,
and write $E_{m,s}=\tilde E_{m}\cap f\inv J_{s}$. For every $m$
let $s(m)$ be such that  
$\abs{ E_{m,s(m)}}\ge \abs{E}\frac{1-\ve}{kl}$,
and set $E_{m}=\tilde E_{m,s(m)}$.
\end{proof}

We refer to this as the \emph{$k,l,\ve$ procedure on $(I;E)$}, call
the pairs $(I_{m};E_{m})$ the (first generation)  descendants
of $(I;E)$ and rename them as $(I_{1,m};E_{1,m})$.

We rename the parameters  $k,l,\ve$ as  $k_{1},l_{1},\ve_{1}$,
and repeat the procedure  on each $(I_{1,m};E_{1,m})$
with parameters  $k_{2},l_{2},\ve_{2}$.
We have the second generation, with $k_{1}k_{2}$ descendants 
named
$(I_{2,m};E_{2,m})$, $m=1,\dots,k_{1}k_{2}$.

We iterate the procedure  repeatedly 
with parameters $k_{j},l_{j},\ve_{j}$ for the $j$'th round, and denote
\begin{equation}
K_{n}=\prod_{j=1}^{n} k_{j},\quad L_{n}=\prod_{j=1}^{n} l_{j}\qua
\and \qua {\tiletab}_{n}=\prod_{1}^{n} (1-\ve_{j}).
\end{equation}

After $n$ iterations we have $K_{n}$ intervals 
$I_{n,m}$, each carrying a subset 
$E_{n,m}$ of $E$ such that
$\abs{E_{n,m}}\ge {\tiletab}_{n}K_{n}\inv L_{n}\inv\abs{E}$,
and any two are separated by intervals of length 
$\ge \ve_{n}\tiletab_{n-1}K_{n}\inv L_{n-1}\inv\abs{E}$. 

\s
Given $\alpha\in (0,1)$, 
we choose the parameters $k_{j}$, $l_{j}$ uniformly bounded,  
and $\ve_{j}\to 0$ so that
\begin{equation}\label{choice}\small
\alpha_{n}=\frac{\log L_{n}}{\log K_{n}+\log L_{n}-\log (\ve_{n}\tiletab_{n})}>\alpha,\quad 
\beta_{n}=\frac{\log K_{n}}{\log K_{n}+\log L_{n}-\log \tiletab_{n}}<1-\alpha,
\end{equation}
and $\alpha_{n}\to\alpha$, $\beta_{n}\to 1-\alpha$.

Denote
$E^{*}_{n}=\cup_{m=1}^{K_{n}}E_{n,m}$, 
observe that 
$E^{*}_{n}\subset E^{*}_{n-1}$, and
  set $E^{*}=\cap E^{*}_{n}$.

We claim that  $E^{*}$ satisfies the requirements of part {\bfit I} of the
theorem. To prove the claim we need to show 

\aba \quad $\hdim E^{*}\ge (1-\alpha)$.   

\bab \quad $f\restr{E^{*}}\in\Lip_{\alpha}$.

\begin{proof} 
For claim \aba we construct a probability
measure $\mu^{*}$ on $E^{*}$, such that  for every $\alpha'>\alpha$,
there exists a  constant $C=C(\alpha')$ such that
$\mu^{*}(I)\le C{\abs I}^{\alpha'}$ for all intervals $I$.
By  lemma \ref{Hausdorff} this proves $\hdim E^{*}\ge (1-\alpha)$.

\s
Denote by $\mu_{n}$ the probability measure obtained by 
normalizing 
the Lebesgue measure on $E_{n}^{*}$ by
multiplying it, on each  $E_{n,m}$, by a constant 
$c_{n,m}=K_{n}\inv\abs{E_{n,m}}\inv$, so that 
$\mu_{n}(E_{n,m})=K_{n}\inv$. 
The sequence $\{\mu_{n}\}$ converges in the weak-star topology
to a measure $\mu^{*}$ carried by $E^{*}$. Observe that 
$\mu^{*}(E_{n,m})=\mu_{n}(E_{n,m})=K_{n}\inv$.

We evaluate the modulus of continuity of the
primitive of $\mu^{*}$ by estimating the size of intervals $A$
such that $\mu^{*}(A)\ge 2 K_{n}\inv$. 
Such interval  must contain an  
interval $I_{n,m}$, and hence $E_{n,m}$, and it follows that
\begin{equation}
\abs{A}\ge\abs{I_{n,m}}\ge\abs{E_{n,m}}\ge  
{\tiletab}_{n}K_{n}\inv L_{n}\inv\abs{E}
\end{equation}
\n
which means that for every $\alpha'>\alpha$ 
we have for $n$ large enough and
every interval $I_{n,m}$  
\begin{equation}
\mu^{*}(I_{n,m})\le\abs{I_{n,m}}^{\frac{\log K_{n}}{\log K_{n}+\log L_{n}-\log  \tiletab_{n}}}=\abs{I_{n,m}}^{\beta_{n}}\le\abs{I_{n,m}}^{1-\alpha'}
\end{equation} and it follows that 
for arbitrary intervals $I$ and any $\alpha'>\alpha$, as $\abs{I}\to 0 $ 
\begin{equation}
\mu^{*}(I)=\osta{\abs{I}^{1-\alpha'}}
\end{equation}
which means that the Hausdorff dimension of $E^{*}$ is at least 
$1-\alpha$.

The modulus of continuity $\vt$ of $f\restr{E^{*}} $ is determined by:

``Let $x,y\in E^{*}$. Let $n$ be the smallest index such 
that $x,y$ are  not in the same $E_{n,m}$. 
Then $\abs{x-y}\ge  \ve_{n}\tilde{\tilde\etab}_{n}K_{n}\inv L_{n}\inv\abs{E}$ 
and $\abs{f(x)-f(y)}\le L_{n-1}^{-1}$.''
which translates to $\vt(\ve_{n}{\tiletab}_{n}K_{n}\inv L_{n-1}\inv\abs{E})\le L_{n-1}\inv$, or, for $t$ in this range
$\vt(t)=\osta{t^{\alpha_{n}}}$, and for all $t$
\begin{equation}
\vt(t)=\osta{t^{\alpha}}
\end{equation}
\save\end{proof}
\begin{remark}Reversing the inequalities in \eqref{choice} 
by an appropriate choice of the parameters
we obtain a set $E^{*}$ that has positive measure in dimension
$1-\alpha$, such that  the modulus of continuity of $f\restr{E^{*}}$ 
is bounded by $t^{\alpha}\abs{\log t}^{\alpha+\ve}$ as $t\to 0$.
\end{remark}

\subsection{Proof of theorem {3.1}, part {\bfit {II.}} }\hspace*{1pt}
As in section 2,  we write 
\begin{equation}\label{lacunar}
f(x)=\sum_{1}^{\infty} a_{j}\vp(\lambda_{j}x), \quad \text{and} \quad
f_{n}(x)=\sum_{1}^{n} a_{j}\vp(\lambda_{j}x)
\end{equation}
where $\vp$ is the 2-periodic function defined by \eqref{varphi},
$a_{j}$ is fast decreasing, $\lambda_{j}$ fast increasing.
Both $a_{j}$ and $\lambda_{j}$ depend on $\alpha$,
and will be defined inductively.

Choose (arbitrarily) $a_{1}=\half$, and $\lambda_{1}=10$.

Assuming $a_{j}$ and $\lambda_{j}$ have been 
chosen for $j\le n$, we
shall choose $a_{n+1}$ 
small (see below) and  then $\lambda_{n+1}$ a large
enough integral multiple of $\lambda_{n}$ so that:
\begin{equation}\label{condition1}
\lambda_{n}\,\mid\,\lambda_{n+1}, \;\;
\text{ and } \;\; a_{n+1}\lambda_{n+1}\ge 2 
\sum_{1}^{n}a_{j}\lambda_{j},
    \end{equation}
The divisibility guarantees that
that $f_{n}$ is linear in each of the intervals
$(\frac{j}{\lambda_{n}},\frac{j+1}{\lambda_{n}})$ and
the successive inequalities in \eqref{condition1} that
$\abs{\frac{d}{dt}f_{n}}\ge \half a_{n}\lambda_{n}>2^{n}$.

\sn Let $E$ be closed, and assume that 
$f\restr E\in\Lip_{\alpha} $.
Denote
\begin{equation*}
E_{n}=\set{x}{x\in E,\qua
\abs{f(x)-f(y)}\le%
n\abs{x-y}^{\alpha}\text{ for all }y\in E\text{ such that }%
\abs{x-y}\le \lambda_{n}\inv}.
    \end{equation*}
Clearly $E_{n}\subset E_{n+1}$, and \ 
$E^{*}=\lim  E_{n}  \supset E$.
 It suffices, therefore,  to show that 
$E_{n}$ can be covered by intervals $I_{j,n}$
such that $\sum_{j} \abs{I_{j,n}}^{\beta}<\ve_{n,\beta}$, 
with $\ve_{n,\beta}\to 0$ for every $\beta>1-\alpha$.

\s
Write $E_{n,j}= E_{n}\cap%
[\frac{j}{\lambda_{n}},\frac{j+1}{\lambda_{n}}]$.
For $x,y\in E_{n,j}$, and in particular the pair $x,y$ 
such that $E_{n,j}\subset [t,y]$ we have
\begin{equation}\label{estim}
n\abs{x-y}^{\alpha}\ge\abs{f(x)-f(y)}\ge
\half a_{n}\lambda_{n}\abs{x-y}-2a_{n+1}.
\end{equation}
If $a_{n+1}$ is small enough, this implies
$\abs{x-y}^{1-\alpha}\le\frac{2n}{ a_{n}\lambda_{n}}$, and 
$E_{n}$ can be covered by $\lambda_{n}$ intervals $I_{j,n}$ of length
$\abs{I_{j,n}}\le\bigbra{\frac{2n}{ a_{n}\lambda_{n}}}^{\frac{1}{1-\alpha}}$.
\s
For any $\beta$,
\begin{equation}\label{}
\abs{I_{j,n}}^{\beta}\le\Bigl(\frac{2n}{ a_{n}\lambda_{n}}\Bigr)^{\frac{\beta}{1-\alpha}}
,\text{\quad  and\quad }\sum\abs{I_{j,n}}^{\beta}\le
\Bigl(\frac{2n}{ a_{n}}\Bigr)^{\frac{\beta}{1-\alpha}}
\lambda_{n}^{1-\frac{\beta}{1-\alpha}}.
    \end{equation}
For $\beta>1-\alpha$ the exponent of $\lambda_{n}$ is negative,
and we take $\lambda_{n}$ big enough (after choosing $a_{n}$).

This concludes the proof of theorem \ref{3.1}.\qed

\section{Lipschitz  and monotone restrictions}

\subsection{Lipschitz restrictions. }

Part {\bfit {II}} of theorem 3 indicates that there 
are continuous functions $f$ such that if $f\restr E\in\Lip_{1}$
then $\hdim E=0$. The following refinement shows that even if $f$ is
``almost'' $\Lip_{1}$, the set $E$ can be ``arbitrarily'' thin. 
\begin{theorem}\label{Lip1}
Given a Hausdorff determining function $h$, 
and a modulus of continuity $\omega$
such that $\lim_{s\to 0} \omega(s)/s=\infty$,
there exist functions $f\in C_{\omega}$ such
that if $f\restr E\in \Lip_{1}$,
then $E$ has zero $h$-measure.
         \end{theorem}
 Notice that the assumption   $\lim_{s\to 0} \omega(s)/s=\infty$,
allows $\omega(s)=\osta{s^{\alpha}}$ for all $\alpha<1$.
The corresponding $f\in C_{\omega}$ belongs to $\Lip_{\alpha}$
for all $\alpha<1$.

\begin{proof} 
We use again the series \eqref{lacunar}, namely
\begin{equation*}
f=\sum_{1}^{\infty} a_{j}\vp(\lambda_{j}x),
\end{equation*}
and adapt the parameters $a_{n}$ and $\lambda_{n}$ to the current
context. Both $a_{j}$ and $\lambda_{j}$ will be defined inductively,
$a_{j}$ will be fast decreasing, $\lambda_{j}$ fast increasing.

Denote by 
$\omega_{n}(s)=\max_{x,\;\abs \tau\le s}a_{n}
\abs{\vp(\lambda_{n}(x+\tau))-\vp(\lambda_{n}(x))} $, 
  the modulus of continuity of
$a_{n}\vp(\lambda_{n}x)$. The condition
$\sum_{n} \omega_{n}(s)=\osta{ \omega(s)}$, as $s\to 0$, 
guarantees that $f\in C_{\omega}$. Observe that
\begin{equation}
\omega_{n}(s)=\min (a_{n}, a_{n}\lambda_{n}s)=
\begin{cases} a_{n}& \text{ if } s>\lambda_{n}\inv\\
a_{n}\lambda_{n}s&\text{if } 0\le s\le \lambda_{n}\inv.
\end{cases}
\end{equation}

\textbf {i.} The first condition we impose on $a_{n},\lambda_{n}$ 
is: $a_{n}\le \omega(1/\lambda_{n}) $.
It implies that 
$\omega_{n}(s)\le\min (a_{n},\omega(s))$ for all $s$.
As $\omega(1/\lambda)>>1/\lambda$, the condition is consistent with having $a_{n}\lambda_{n}$ 
arbitrarily large.

\textbf{ ii.} Given $a_{n}$ and $\lambda_{n}$,  
define  $c_{n}$ by the condition 
$\omega(c_{n})= 2^{n}a_{n}\lambda_{n}c_{n}=%
2^{n}\omega_{n}(c_{n})$. This implies 
that    for $s\le c_{n}$ we have
$\omega(s)\ge 2^{n}a_{n}\lambda_{n}s$  and 
\begin{equation}
\omega_{n}(s)\le
\begin{cases} a_{n}& \text{ if } s>c_{n}\\
2^{-n}\omega(s)&\text{if } s\le c_{n}.
\end{cases}
\end{equation}
so that  
for  $c_{n+1}\le s\le c_{n}$  we have  
$\sum \omega_{j}(s)\le \omega(s)+\sum_{j=n+1}^{\infty} a_{j}$.
It follows that 
 if $a_{n}$ decreases fast enough  (while $\lambda_{n}$ increases, 
 allowing $a_{n}\lambda_{n}$ to be as large as is needed),
 we have indeed $f\in  C_{\omega}$. 

\textbf{iii.} Assuming $a_{j}$ and $\lambda_{j}$ have been 
chosen for $j\le n$, we
shall choose $a_{n+1}$ 
small (see below) and  then $\lambda_{n+1}$ a large
enough integral multiple of $\lambda_{n}$ so that:
\begin{equation}\label{condition}
\lambda_{n}\,\mid\,\lambda_{n+1}, \;\;
\text{ and } \;\; a_{n+1}\lambda_{n+1}\ge 2 
\sum_{1}^{n}a_{j}\lambda_{j},
    \end{equation}
The divisibility guarantees that
that $f_{n}$ is linear in each of the intervals
$(\frac{j}{\lambda_{n}},\frac{j+1}{\lambda_{n}})$ and
the successive inequalities in \eqref{condition} that
$\abs{\frac{d}{dt}f_{n}}\ge \half a_{n}\lambda_{n}>>2^{n}$.

Let $E$ be closed, and assume that $f\restr E\in\Lip_{1} $.
Denote
\begin{equation*}
E_{n}=\set{x}{x\in E,\qua \abs{f(x)-f(y)}\le%
n\abs{x-y}\text{ for all }y\in E\text{ such that }%
\abs{x-y}\le \lambda_{n}\inv}.
    \end{equation*}
Clearly $E_{n}\subset E_{n+1}$, and \ 
$E^{*}=\lim  E_{n}  \supset E$.
 It suffices, therefore,  to show that 
$E_{n}$ can be covered by intervals $I_{j,n}$
such that $\sum_{j} h(\abs{I_{j,n}})<\ve_{n}$, with $\ve_{n}\to 0$.

\s
Write $E_{n,j}= E_{n}\cap%
[\frac{j}{\lambda_{n}},\frac{j+1}{\lambda_{n}}]$.
If $x,y\in E_{n,j}$ then
\begin{equation}\label{estim}
n\abs{x-y}\ge\abs{f(x)-f(y)}\ge
\half a_{n}\lambda_{n}\abs{x-y}-2a_{n+1}
\end{equation}
which implies
$\abs{x-y}\le 4a_{n+1}/(a_{n}\lambda_{n}-2n)$.
It follows that $E_{n}$ can be covered by $\lambda_{n}$ arcs
of length bounded by 
$l_{n}=4a_{n+1}/(a_{n}\lambda_{n}-2n)<%
5a_{n+1}/a_{n}\lambda_{n}$.

Choose $a_{n+1}$ small enough so  that
$\lambda_{n}h(l_{n})< n^{-n}$, %
and then $\lambda_{n+1}$ appropriate to guarantee \eqref{condition}.
\remark The proof shows, in fact, that $E$ is Minkowski $h$-null.
\end{proof}
\subsection{Monotone restrictions. }
Does there exist a  function $f\in C([0,\,1)$  such that 
if $f\restr{E}$ is monotone then $E$ has Hausdorff dimension $0$?

\begin{theorem}
Given a Hausdorff
determining function $h$, there exists  $f\in C([0,\,1])$ such
that if $f\restr{E}$ is monotone,
then $E$ has zero $h$-measure.
\end{theorem}
\begin{proof}
Now we have to
give up the building block  $\vp$  defined in  (9) and the corresponding
functions  $\vp_{n}$.  Let us denote
 by $\psi_{m}(x)$ the 1-periodic function satisfying:  $\psi_{m}(0)=\psi(1)=0$, $\psi_{m}(m\inv)=1$ and
$\psi_{m}(x)$ linear on
$[0,\,m\inv]$ and on $[m\inv,\,1]$.

Write $f=\sum_{1}^{\infty} a_{j}\psi_{m_{j}}((-1)^{j}\lambda_{j}x)$
and $f_{n}=\sum_{1}^{n} a_{j}\psi_{m_{j}}((-1)^{j}\lambda_{j}x)$,
where $a_{j}$, $m_{j}$,
and $\lambda_{j}$ will be defined inductively.

The first conditions are
\begin{equation}\label{condition3}
m_{j-1}\lambda_{j-1}\,\mid\, m_{j}\lambda_{j}, \;\;
\text{ and } \;\; a_{n}\lambda_{n}\ge 
2 \sum_{1}^{n-1}a_{j}m_{j}\lambda_{j},
    \end{equation}
so that $f_{n}$ is linear in each of the intervals ($n$-intervals)
$(\frac{j}{m_{n}\lambda_{n}},\frac{j+1}{m_{n}\lambda_{n}})$.
Each such
 \linebreak interval is divided in the next generation into one ``fast''
interval on which\linebreak
$\abs{\frac{d}{dt}f_{n+1}}\sim a_{n+1} m_{n+1}\lambda_{n+1}$
and the union of the remaining  ``slow'' intervals on which
$\abs{\frac{d}{dt}f_{n+1}}\sim a_{n+1}\lambda_{n+1}$.

For even $n$ (resp. odd $n$) $f_{n}$ is increasing 
(resp. decreasing) on the fast intervals and decreasing  (resp. increasing) on the unions of the slow ones contained 
in an $(n-1)$-interval.

\s
Let $E$ be closed, $f\restr{E}$ monotone increasing.
Let $n$ be even.
Then, if $J$ is the slow part of an $n$-interval, the diameter
of $J\cap E$ is bounded by $a_{n+1}/a_{n}\lambda_{n}$.
The number of such $J$'s is $\lambda_{n}$.
Choose $a_{n+1}$ such that 
$\lambda_{n} h(a_{n+1}/a_{n}\lambda_{n})\to 0$.

 $E\setminus \bigcup J $ is covered by the union 
 of the fast $n$-intervals
that is $\lambda_{n}$ intervals of length $m_{n}\inv$.
Choose $m_{n}$ (after choosing $\lambda_{n}$) 
so that $\lambda_{n}h(m_{n}\inv)\to 0$.
\end{proof}

\section{Restrictions of H\"older functions}
\subsection{Smoothness. }
\begin{theorem}
Assume that $0<\beta<\alpha<1$. 
There exist functions 
$f\in\Lip_{\beta}$
such that if $f\restr E\in \Lip_{\alpha}$, 
then $E$ has Hausdorff dimension bounded by 
$\frac{1-\alpha}{1-\beta}$.
\end{theorem}
\begin{proof}
We keep the notations used in the proof of theorem \ref{Lip1}. As
observed there,
the condition $f\in\Lip \beta$ is equivalent to
 $a_{n}=\ostaB {\lambda_{n}^{-\beta}}$ (if $\lambda_{n}$ grows 
fast enough). 
 Now $a_{n}^{-\frac{\alpha'}{1-\alpha}}
\lambda_{n}^{1-\frac{\alpha'}{1-\alpha}}=\ostaB{\lambda_{n}^{-\beta\frac{\alpha'}{1-\alpha}+1-\frac{\alpha'}{1-\alpha}}}$ 
and the exponent is negative if $\alpha'>\frac{1-\alpha}{1-\beta}$.
 \end{proof}

\b
\n
\bfit{ Question. }Is the following statement valid?

\emph{
Assume $0<\beta<\alpha<1$. If $f\in\Lip_{\beta}$ there exists a set
$E$ such that $\hdim E=\frac{1-\alpha}{1-\beta}$, and 
$f\restr E\in \Lip_{\alpha}$.}

\subsection{Bounded variation. }
For $\alpha\in (0,1)$, denote by 
$\norm{\ }_{\alpha}$ the $\Lip_{\alpha}$ norm.
It is easy to see that
$\norm{a\vp_{n}}_{\alpha}\sim an^{\alpha}$ and if $n_{k}$
increases fast enough, say $n_{k+1}>2 n_{k}$,
then $\sum a_{k}\vp_{n_{k}}\in\Lip_{\alpha}$ if, and only if,
 $a_{k}=\osta{n_{k}^{-\alpha}}$.
\begin{theorem}
 There exists real-valued  functions $F\in \Lip_{\alpha}$
such that if $E\subset [0,1]$ is closed and 
$\lmdim(E)> \frac{1}{2-\alpha}$
then $\var(E,F)=\infty$. 
\end{theorem}
\begin{proof}
As in the example above 
define $F=\sum a_{k}\vp_{n_{k}}$ where now
 $n_{k}={a_{k}^{-1/\alpha}}$. 
If $\lmdim(E)> \frac{1}{2-\alpha}$, and we set
  $s_{k}=n_{k}^{\alpha-2}\log n_{k}$,
 then
 $E$ contains  $s_{k}$-separated sequences $J'_{k}$ of length 
 $m_{k}>20 n_{k}$, and $\var(E,F)=\infty$ since for every $k$,
\begin{equation}
\var(E,F)\ge \var(J'_{k},F)\ge n_{k}^{2}a_{k}s_{k}=\log n_{k}.
\end{equation}\save
\end{proof}

\bfit {Question:} Is the result best possible: does every 
$f\in \Lip_{\alpha}$ have bounded variation on 
some set of dimension 
$c=\frac{1}{2-\alpha}$?

\end{document}